\documentclass[envcountsect]{svjour3}  

\usepackage{graphicx} 
\usepackage{subfig}
\usepackage{color}
\usepackage{etoolbox}
\usepackage{amsmath}
\usepackage{amsfonts}
\usepackage{amssymb}
\usepackage{url}
\usepackage{marginnote}
\usepackage{algorithm}
\usepackage{algorithmic}

\usepackage{tabularx}
\smartqed 
\usepackage{graphicx}
\usepackage[numbers]{natbib}

\journalname{LION19}

\newtheorem{thm}{Theorem}
\newtheorem{cor}[thm]{Corollary}
\newtheorem{lem}[thm]{Lemma}

\newtheorem{assume}[thm]{Assumption}

\DeclareMathOperator*{\argmin}{argmin}

\newcommand{\R}{\mathbb{R}}

\begin{document}

\title{Time-Varying Multi-Objective Optimization: \\ Tradeoff Regret Bounds}


\author{Allahkaram Shafiei, Jakub Marecek }

\institute{Allahkaram Shafiei and Jakub Marecek \at
             Department of Computer Science, Czech Technical University in Prague \\
              Prague, the Czech Republic
}

\date{Received: date / Accepted: date}

\maketitle

\begin{abstract}
Multi-objective optimization studies the process of seeking multiple competing desiderata in some operation. Solution techniques highlight marginal tradeoffs associated with weighing one objective over others. In this paper, we consider time-varying multi-objective optimization, in which the objectives are parametrized by a continuously varying parameter and a prescribed computational budget is available at each time instant to algorithmically adjust the decision variables to accommodate for the changes.
We prove regret bounds indicating the relative guarantees on performance for the competing objectives.
\end{abstract}
\keywords{Time-varying\and Multi-objective \and Regret Bound\and Proximal Gradient Descent\and }


\section{Introduction and Preliminaries}

During the past two decades, substantial research efforts have been devoted to learning and decision-making in environments with functionally relevant data arriving in a streaming fashion \cite{{Zinkevich2003,Hazan2016,9133310}}, with potentially changing statistical properties. 
In terms of guarantees,
relatively little is known 
about optimization algorithms \cite[e.g.]{goh2009evolutionary,Raquel2013,chen2019novel,9440867,liu2023survey,10644089} considering  multiple criteria and data arriving in a streaming fashion, so far.

Notice that multiple criteria and data arriving in a streaming fashion are a very common setting. 
Indeed, there are both multi-objective and time-varying aspects involved in most engineering problems.   
For example, in trajectory planning for autonomous driving \cite{li2023trajectory},  there are multiple criteria (e.g., safety, fuel efficiency, environmental impact) and time-varying inputs (e.g., from cameras and lidar sensors). 
In design problems in power engineering \cite{Baran1989,7540906}, 
there are multiple criteria, but 
the loads in a power system could vary over time \cite{8442544}, both over slow time scales (e.g., due to migration, relevant in design problems) and faster time scales (less irrelevant in design problems). 
In topology optimization for structural engineering and materials engineering \cite{Marler2004,Ramamoorthy2023}, there are potentially multiple criteria (cost, performance, environmental impact) and forces acting upon a structure could vary over time \cite{Ma2025}.
In coordination of fire-fighting mobile robots \cite{Derenick2007}, there are multiple criteria related to the wild fire, but the wildfire spreads across a country in response to wind, which may be hard to predict.
In machine-learning applications of multi-objective optimization, the time-varying aspects could capture, e.g., time-varying group structure, seasonal or circadian cyclicity, or some form of a concept drift. In game theory, the time-varying aspects could capture time-varying pay-offs (or time-varying price elasticity of the demand) in extensive forms of Stackelberg-like games or time-varying demands in congestion games.

The setting of a dynamically changing and uncertain environment can be analyzed in the framework of online optimization \cite[e.g.]{Zinkevich2003,Hazan2016,9133310},
where the cost function changes over time and an adaptive
decision pertaining only to past information has to be
made at each stage. 
The standard guarantees in online optimization include the level of \emph{regret}, a quantity capturing the difference between the accumulated cost incurred up to some arbitrary time and
the cost obtained from the best fixed point (or some times, policy within a class) chosen in hindsight. We develop such guarantees for time-varying multi-objective optimization.  

Our contributions include:
\begin{itemize}
\item introduction of regret tradeoffs as the appropriate metric for grading solvers for online multi-objective optimization
\item  an on-line proximal-gradient algorithm for handling multiple time-varying convex objectives, which is amenable to analysis, 
\item  theoretical guarantees for the algorithm.
\end{itemize}

\section{Related work}
Proximal Gradient Descent (PGD) is a natural approach to minimize single and multiple objectives. One of the most studied methods for multiobjective optimization problems is steepest descent, for example \cite{Geoffrion1968,Goncalves2022}.
  Subsequently, a proximal point method \cite{Bonnel2005}, which can be applied
to non-smooth problems, was considered. However, this method is just a conceptual scheme and does not necessarily generate subproblems that are easy to solve.
For non-smooth problems, a subgradient method was also developed \cite{Bagirov2013}. A very comprehensive recent paper~\cite{Zhou2020} has presented the regret bounds for classic algorithms for online convex optimization with Lipschitz, but possibly non-differentiable functions, proving a regret of $O(\frac{1}{\sqrt{K}})$, with $K$ iterations at each time instant. With respect to multiobjective (but not online) optimization, Tanabe et al.~\cite{Tanabe2019} proposed proximal gradient methods with and without line searches for unconstrained multiobjective optimization problems, in which every objective function is of the composite form of interest in our work,
$F_i(x)=f_i(x)+g_i(x)$, with $f_i$ smooth and $g_i$ merely proper and convex, but with a tractable proximal computation. 

Next, we describe the literature on online time-varying convex single-objective optimization. As the first innovative paper in this space, Zinkevich \cite{Zinkevich2003} proposed a gradient
descent algorithm with a regret bound of $O(\sqrt{K})$. In the case where the cost functions are strongly convex, the regret limit of the online gradient descent algorithm was further reduced to $O(\log(K))$ with the suitable step size chosen by several online algorithms presented in \cite{Hazan2007}.

 \section{Problem Formulation}
 We begin with describing the problem of Time-Varying Multi-Objective Optimization. Suppose that we have a sequence of convex cost
functions $\phi_{i,t}(x):=f_{i,t}(x)+g_{i,t}(x)$ where $f_{i,t}$ are smooth and $g_{i,t}$ non-smooth. The index $t$ corresponds to the time step and $i$ indexes the objective function among the set of desiderata. 
  Between two time steps $t\in [T]:=\{1,2,...,T\}$, there is a finite amount of time available to compute an optimal decision, suggesting an upper bound on the iteration count for any algorithm, at which point the decision maker must choose an action $x^t\in\mathbb{R}^N$, and the decision maker is faced with a loss of $\phi_{i,t}(x^{t})$. In this scenario, due to insufficient computation time, the decision does not necessarily correspond to the minimizers, and the decision maker faces a so-called regret. Regret is defined as the difference between the accumulated cost over time and the cost incurred by the best-fixed decision when all functions are known in advance, see \cite{Zinkevich2003,Hazan2007,Hazan2017}. Let us consider $\phi_{i,t}(x)=f_{i,t}(x)+g_{i,t}(x)$ as
 \[
F(x,t) = \left(\phi_{1,t}(x),\phi_{2,t}(x),...,\phi_{M,t}(x)  \right)
 \]
 At the time $t$, we consider the following  time-varying vector optimization 
 \begin{equation}\label{eq:prob1}
     \min\limits_{x\in\mathbb{R}^N} \,\left( F(x,t):=F_t(x)
\right) \end{equation}
 where $F:\mathbb{R}^N\to \mathbb{R}^M$ and, each $f_{i,t}$ is $L_{f_{i,t}}$ Lipschitz continuously differentiable and $g_i$ is convex with a simple prox evaluation. Throughout the paper, our discussion of \eqref{eq:prob1} and the proposed algorithm are motivated by the following:
 As can be surmised from the definition, there is rarely a singleton that is a Pareto optimal point. Usually, there is a continuum of solutions. As such, one can consider a \emph{Pareto front} which indicates the set of Pareto optimal points. The front represents the objective values reached by the components of the range of $F(x)$ and it is usually a surface of dimension $m-1$. One can consider it as representing the tradeoffs associated with the optimization problem, to lower $i$'s value, i.e. $f_i(x)$, how much are you willing to compromise in terms of potentially raising $f_j(x)$ for the set of $j\neq i$? 
 
 Because of the generality of the concept of a solution to a vector optimization problem \cite{jahn2010vector}, there are a variety of problems associated, including visualizing the entire Pareto front, or some portion of it, finding \emph{any} point on the Pareto front, or finding some point that satisfies an additional criteria, effectively making this a bilevel optimization problem. 
  In regards to the second option, one can notice that in the convex case, this can be done by solving the so-called ``scalarized" problem:
 \[
 \min\limits_{x\in\mathbb{R}^N} \sum\limits_{i=1}^n \omega_i f_i(x),\,\sum\limits_{i=1}^n \omega_i=1, \,0\le \omega_i\le 1
 \]
 for any valid choice of $\{\omega_i\}$. This reduces the problem to simple unconstrained optimization. This leaves the choice of said constants, however, arbitrary, and thus not all that informative. Although the parameters are weights balancing the relative importance of the objective functions, poor relative scaling across $f_i(x)$ can make an informed choice of $\{\omega_i\}$ insurmountable. For example, if $f_1(x)=1000 x^2$ and $f_2(x) = 0.001 (x-2)^2$, taking $\omega_1=\omega_2=0.5$ clearly pushes the solution of the scalarized problem to prioritize minimizing $f_1(x)$. 
 
 As an additional challenge, we consider the time-varying case, i.e., each $f_i(x)$ changes over time, e.g., due to data streaming with concept drift. With a finite processing capacity at each time instant, we seek an Algorithm that appropriately balances the objectives at each time instant. 
 
 In this paper, we introduce scalarization at the \emph{algorithmic} level for time-varying multi-objective optimization. In particular, at each iteration, we consider computing a set of steps, each of which intends to push an iterate towards the solution of the problem of minimizing $f_i(x)$ exclusively. The algorithm then forms a convex combination of these steps with \emph{a priori} chosen coefficients. We derive \emph{tradeoff regret bounds} indicating how the choice of the said coefficients results in guarantees in regard to suboptimality for every objective. We assert that this would be the most transparently informative theoretical guarantee, in terms of exactly mapping algorithmic choices to comparative performance for every objective function, and as such a natural and important contribution to time-varying multi-objective optimization.


Now we shall present our formal assumptions regarding the problem, in particular the functional properties of $F$ as well as the algorithm we propose and study the properties of in this paper. 
 \begin{assume}[Problem Structure]
 \label{ass1}
\begin{itemize}
	 		\item[$\rm{(i)}$] For all $i,t$ functions $f_{i,t}(\cdot) : \mathbb{R}^N \rightarrow \mathbb{R}$ are continuously 
	 			differentiable such that the gradient is Lipschitz with constant $L_{f_{i,t}}$: \begin{equation*}
	 				\|\nabla f_{i,t} (x) - \nabla f_{i,t}(y)\| \leq L_{f_{i,t}}\|x-y\|,\,\,\,\,\forall x,y\in\mathbb{R}^N.
	 			\end{equation*}
	 		\item[$\rm{(ii)}$] For all $t$, the function $g_{i,t}(\cdot) : \mathbb{R}^N \rightarrow (-\infty,\infty]$ is proper, 
	 			lower semi-continuous, and convex, but not necessarily differentiable. Also, assume that ${\rm dom}(g_{i,t}(\cdot))=\{x\in\mathbb{R}^N:~g_{i,t}(x)<\infty\}$ is non-empty and closed. 
	 			\item [$\rm{(iii)}$] Fo each objective $\phi_{i,t}$, we consider $T_{i,t}(x)={\rm prox}_{C_ig_{i,t}}(x-C_i\nabla f_{i,t}(x))$.
	 	\end{itemize} 
 	\end{assume} 
  We also assume a bound on the magnitude of change between successive times:
 	
 \begin{assume}[Slow Changes] \label{as:varying} 
 The observations as compared to estimates of the function values from the previous time step are bounded at all $x$, i.e.,
 \[
 \mathop {\sup }\limits_{t\ge 1} \max_{i\in[n]} \left\{ \begin{array}{l}
 |f_{i,t+1}(x) - f_{i,t} (x)|,\\
 |g_{i,t+1}(x) - g_{i,t} (x)| 
 \end{array} \right\} \le e
 \]
 \end{assume}

  \section{The Algorithm and Preliminaries}
 An on-line proximal-gradient algorithm is stated formally as Algorithm~\ref{alg:PoEG}. The coefficients $\{\alpha_i\}$ denote the priority of objective $i$, and belong to the unit simplex (i.e.,  $\sum_{i=1}^n \alpha_i=1$, $0\le \alpha_i\le 1$). 
  The following assumptions are typical in the analysis of online algorithms and make real-time algorithmic path-following of solutions feasible.
   In particular, we consider the on-line streaming setting with a finite sampling rate, which we assume permits $K$ iterations of the proximal-gradient steps between two updates of the inputs. 
  
 \begin{assume}[Sufficient Processing Power] \label{as:timeK} 
 At all times $t\in[T]$, the algorithm executes at least $K$ iterations before receiving the new input.
 \end{assume}



We consider two measures of the quality of the solution trajectory:
\begin{enumerate}
\item[$(A)$] 
As a variant of dynamci regret bound of \cite{Zinkevich2003}, we define the dynamic regret bound for the convex combination of $\phi_{i,t}$ as follows:

\[{\mathop{\rm Re}\nolimits} {g_t} = \sum\limits_{t = 1}^T {\phi_{t} ({x^t}) - } \sum\limits_{t= 1}^T {\phi_t(x^{{\rm opt},t})}.\]
In the case of static regret \cite{Hazan2017}, $x^{{\rm opt},t,i}$ is replaced by  $x^{{\rm opt},i}\in\argmin_{x\in X}\sum_{t=1}^T\phi_{i,t}(x)$, i.e, 
\[{\rm S-Reg}_i=\sum_{t=1}^T\phi_{i,t}(x^t)-\min_{x\in X}\sum_{t=1}^T\phi_{i,t}(x)\]
\item[$(B)$] In addition, we will consider the following quantities:
\[\begin{array}{l}
 \begin{array}{*{20}{c}}
   \hfill  \\
   {} \hfill & {\,{\phi _t}(x) := \mathop\sum\limits_{i\in[n]}\alpha_i\phi_{i,t}(x),\,\,\,\,\,\,\,{x^{{\rm{opt}},t}}\in \mathop{\rm argmin}\limits_{x}\phi_t(x)\,\,\,\,} \hfill  \\
\end{array} \\ 
 \,\,\,\,{W_T} := \sum\limits_{t\in[T]}\|x^{\rm{opt},t + 1} - x^{\rm{opt},t}\|^2
 \end{array}.\]
\end{enumerate}


\begin{algorithm}[t!]
    \caption{On-Line Multi-Objective Proximal Gradient Descent}
    \textbf{Input:} Initial iterate $x^1$ solving the problem with data $f_{1,1}(x),g_{1,1}(x)$ parameters $C_1\in(0,\frac{1}{L_{f_{1,1}}}]$, $\alpha_i>0$, and let $x^{1,0}\gets x^1$

    \textbf{\bf for $t=1,2,...,T$ do}  
    
     \quad $x^{t,1}\gets x^{t}$;
     
     \quad 
    Receive data $f_{i,t}(x^t),g_{i,t}(x^t)$;
    
    \quad {\bf for $k=0,1,2,...,K$ do} 
  
  \hspace{1cm} $y^{t,k+1,i}\gets {\rm prox}_{C_{i} g_{i,t}}(x^{t,k}-C_{i}\nabla f_{i,t}(x^{t,k}) )\, \forall i$;\\
  
  \hspace{1cm} $x^{t,k+1}\gets\sum\limits_{i = 1}^n {{\alpha _i}{y^{t,k+1,i}}}$;
  
  \hspace{1cm} $k\gets k+1,~t\gets t+1$
  
   \quad {\bf end for } 
   
  \quad $x^{t+1,0}\gets x^{t,K}$ and $x^{t+1}\gets x^{t,K+1}$; 
  
    {\bf end for} 
    \label{alg:PoEG}
\end{algorithm}

 The following lemma is a key result throughout the paper.
 \begin{lemma}\label{eq:lemmaestimate}
Let $f$ be convex and smooth, and $g$ be non-smooth and $\phi=f+g$ then 
\begin{align}
   & \phi(T(x))-\phi(y)\le \frac{1}{2C}[\|x-y\|^2-\|T(x)-y\|^2]\\& {\rm and} ~~ \phi(T(x))\le \phi(x).
\end{align}
where $T(x)={\rm prox}_{Cg}(x-C\nabla f(x))$, $C\in(0,\frac{1}{L_f}]$ and $L_f$ is Lipschitz constant for $\nabla f$.
 \end{lemma}
 \begin{proof}
 Take $G(x):=\frac{1}{C}(x-T(x))$ and apply the standard Descent Lemma:
 \begin{align}\label{eq:r1}
     f(y)\le f(x)+\nabla f(x)^T(y-x)+\frac{L_f}{2}\|x-y\|^2~~~\forall x,y\in\R^N.
 \end{align}
 Plugging $y=x-CG(x)$ in \eqref{eq:r1} one obtains that \begin{align}
     f(x-CG(x))&\le f(x)+\nabla f(x)^T((x-CG(x))-x)+\frac{L_fC^2}{2}\|G(x)\|^2\\
     &\le f(x)-C\nabla f(x)^T(G(x))+\frac{C}{2}\|G(x)\|^2.
 \end{align}
 Now, from $x-CG(x)={\rm prox}_{Cg}(x-C\nabla f(x))$ it follows that \[G(x)-\nabla f(x)\in\partial g(x-CG(x)).\]
 Therefore, for any $y$, by convexity of $g$ we obtain the relation:
 \begin{align}\label{eq:r2}
     g(x-CG(x))\le g(y)-\left(G(x)-\nabla f(x)\right)^T(y-x-CG(x)).
 \end{align}
 Now consider $\phi(T(x))=\phi(x-CG(x)$. By simplifying and applying \eqref{eq:r2}, one has 
 \begin{align*}
     \phi(x-CG(x))&=f(x-CG(x))+g(x-CG(x))\\&\le f(x)-C\nabla f(x)^T(G(x))+\frac{C}{2}\|G(x)\|^2+g(x-CG(x))\\&\le f(y)-\nabla f(x)^T(y-x)-C\nabla f(x)^T(G(x))+\frac{C}{2}\|G(x)\|^2+g(x-CG(x))\\&
     \le f(y)-\nabla f(x)^T(y-x)-C\nabla f(x)^T(G(x))\\&+\frac{C}{2}\|G(x)\|^2+g(y)-(G(x)-\nabla f(x))^T(y-x+CG(x))\\&=
     \phi(y)-\nabla f(x)^T(y-x)-C\nabla f(x)^T(G(x))+\frac{C}{2}\|G(x)\|^2-G(x)^T(y-x)\\&-C\|G(x)\|\nabla f(x)^T(y-x)+C\nabla f(x)^T(G(x))
     \\&\le
     \phi(y)+\frac{1}{2C}[\|x-y\|^2-\|(x-y)-CG(x)\|^2],
 \end{align*} 
 \end{proof}

 \section{ Main Results}
 
 Our main result provides a bound on the expected dynamic regret of the online multi-objective proximal gradient descent (Algorithm 1). Depending on the coefficients $\alpha_i$, there are two cases:
 \begin{enumerate} 
 \item[i)] if for all $i\in[n]$, $\alpha_i\ne0$ and 
 \item[ii)] if there is $i\in[n]$ such that $\alpha_i=1$ and for all $j\ne i$, $\alpha_j=0$. 
 \end{enumerate}
 For case i), we have:
\begin{thm}\label{eq:maintheorem}
 Let $x^t$, (t=1,..., T) be a sequence generated by running Algorithm 1 over $T$ time steps. Then, we have 
 \[Reg = \sum\limits_{t = 1}^T {{\phi _t}({x^t}) - \sum\limits_{t = 1}^T {{\phi _t}} ({x^{{\rm{opt}},t}})}  \le C_T+4(T-1)e+\|x^1-x^{\rm opt,1}\|^2+W_T\]
 where $C_T=\lvert\phi_1(x^1)-\phi_1(x^{\rm opt,T})\rvert$. In addition one has,
 \[Reg = \sum\limits_{t = 1}^T {{\phi _{i,t}}({x^t}) - \sum\limits_{t = 1}^T {{\phi _{i,t}}} ({x^{{\rm{opt}},t}})}\le\frac{1}{\alpha_i}\left[C_T+4(T-1)e+\|x^1-x^{\rm opt,1}\|^2+W_T\right]. \]
\end{thm}
To prove the result, we need a technical lemma:
 \begin{lem}\label{eq:lemma}
 The following holds:
 \begin{enumerate}
\item[a)] For all $t\in [T], k\in[K]$ one has
 \[\|x^{t,k+1}-x^{\rm opt, t}\|\le\|x^{t,k}-x^{\rm opt, t}\|.\]
\item[b)] For all $t\in[T]$ one has $\phi_t(x^{t+1})\le\phi_t(x^t)$ and particularly \[\lvert \phi_t(x^{t+1})-\phi_{t+1}(x^t)\rvert<e.\]
 \item[c)] For all $t\in[T]$, one has $\lvert\phi_t(x)-\phi_{t+1}(x)\rvert<2e$. 
 \end{enumerate}
 \end{lem}
 
 Returning to the proof of the main result, 
\begin{proof}  
    
Utilizing Lemma \ref{eq:lemmaestimate}, one obtains
    \begin{align}
        & \phi_t(x^{t+1})-\phi_t(x^{\rm opt,t}) \nonumber \\
        & =\phi_t(T_t(x^{t,K}))-\phi_t(x^{\rm opt,t})\le\frac{1}{\widetilde C}\left[\|x^{t,K}-x^{\rm opt,t}\|^2-\|T_t(x^{t,K})-x^{\rm opt,t}\|^2\right] \nonumber \\&
        \nonumber=\frac{1}{\widetilde C}\left[\|x^{t,K}-x^{\rm opt,t}\|^2-\|x^{t+1}-x^{\rm opt,t}\|^2\right]\le \frac{1}{\widetilde C}\left[\|x^{t,1}-x^{\rm opt,t}\|^2-\|x^{t+1}-x^{\rm opt,t}\|^2\right]\\&\label{eq:1}
        =\frac{1}{\widetilde C}\left[\|x^t-x^{\rm opt,t}\|^2-\|x^{t+1}-x^{\rm opt,t}\|^2\right]
    \end{align}
    where $T_t(x)=\sum_{i\in[n]}\alpha_iT_{i,t}(x),~ \widetilde C=\sum_{i\in[n]}\alpha_iC_{i}$.
    Alternatively, it is straightforward to verify that  
    \begin{align}\label{eq:inequlity}
        \|x^{t+1}-x^{\rm opt,t}\|^2\ge \|x^{t+1}-x^{\rm opt,t+1}\|^2-\|x^{\rm opt,t+1}-x^{\rm opt,t}\|^2
    \end{align}
    the above combined with \eqref{eq:1} leads to the following 
    \begin{align*}
        \phi_t(x^{t+1})-\phi_t(x^{\rm opt,t})\le\frac{1}{\widetilde C}\left[\|x^t-x^{\rm opt,t}\|^2-\|x^{t+1}-x^{\rm opt,t+1}\|^2+\|x^{\rm opt,t+1}-x^{\rm opt,t}\|^2\right].
    \end{align*}
    Now, summing up the result over  $t\in[T]$ obtains 
    \begin{align}\label{eq:maininequ}
        \sum_{t=1}^T\phi_t(x^{t+1})-\sum_{t=1}^T\phi_t(x^{\rm opt,t})\le\frac{1}{\widetilde C}\left[\|x^1-x^{\rm opt,1}\|^2+W_T\right]
    \end{align}
    which resulted in
    \begin{align}
         \sum_{t=1}^{T-1}\phi_t(x^{t+1})-\sum_{t=1}^{T-1}\phi_t(x^{\rm opt,t})\le \frac{1}{\widetilde C}\left[\|x^1-x^{\rm opt,1}\|^2+W_T\right]
    \end{align}
    on the other side since $\phi_{t+1}(x^{t+1})-2e\le\phi_t(x^{t+1})$, 
    we would have: 

    \begin{align}
    \left ( \sum_{t=1}^T\phi_t(x^t)- \sum_{t=1}^T\phi_t(x^{\rm opt,t})\right)-\phi_1(x^1)-2(T-1)e+\phi_T(x^{\rm opt,T}) \notag \\ \le \frac{1}{\widetilde C}\left[\|x^1-x^{\rm opt,1}\|^2+W_T\right].
    \end{align}
    Summation over $t\in [T]$ yields 
    \begin{align}
        \sum_{t=1}^T\phi_t(x^t)- \sum_{t=1}^T\phi_t(x^{\rm opt,t})\le C_T+4(T-1)e+\frac{1}{\widetilde C}\left[\|x^1-x^{\rm opt,1}\|^2+W_T\right].
    \end{align}
\end{proof}

In the following corollary, it will be shown that for a single objective, the dynamic regret bound is weaker than for a multi-objective case. 
\begin{cor}
In the case that there exists $i$,  $\alpha_i=1$ and for all $j\ne i$ we have $\alpha_j=0$ the problem reduces to time-varying single objective optimization, i.e, 
 $$x^{t,k+1}={\rm prox}_{C_ig_{i,j}}(x^{t,k}-C_i\nabla f_{i,j}(x^{t,k})).$$ 
 Then 
 \begin{align}
\sum_{t=1}^{T}\phi_{i,t}(x^{t})-\phi_{i,t}(x^{{\rm opt},t,i})\le C_T+4(T-1)e+\frac{1}{(K+1)C_i}\left[\|x^1-x^{\rm opt,1}\|^2+W_T\right]. \notag
 \end{align}
 \end{cor}
 \begin{proof}
As can be seen from Lemma \ref{eq:lemmaestimate} we have 
\[\phi_{i,t}(x^{t,k+1})-\phi_{i,t}(x^{{\rm opt},t,i})\le \frac{1}{C_i}[\|x^{t,k}-x^{{\rm opt},t,i}\|^2-\|x^{t,k+1}-x^{{\rm opt},t,i}\|^2].\]
Summing the result over $k$ from $1$ to $K$, we conclude that 
\[\sum_{k=1}^{K}[\phi_{i,t}(x^{t,k+1})-\phi_{i,t}(x^{{\rm opt},t,i})]\le\frac{1}{C_i}[\|x^{t,1}-x^{{\rm opt},t,i}\|^2-\|x^{t,K+1}-x^{{\rm opt},t,i}\|^2], \]
Since $\phi_{i,t}(x^{t,k+1})\le\phi_{i,t}(x^{t,k})$, the previous term becomes 
\[\phi_{i,t}(x^{t,K+1})-\phi_{i,t}(x^{{\rm opt},t,i})\le\frac{1}{(K+1)C_i}[\|x^t-x^{{\rm opt},t,i}\|^2-\|x^{t+1}-x^{{\rm opt},t,i}\|^2].\]
Now, by using inequalities \eqref{eq:inequlity} and \eqref{eq:maininequ}, and subsequently summing the previous inequality over $t$ from 1 to $T$, one establishes the required assertions.
 \end{proof}

If we make an additional assumption regarding the correspondence of the function values and the distance to the solution set, we can obtain guarantees for the latter.
\begin{assume}
   For all $i$ and $t$, $\phi_{i,t}$ satisfies the quadratic
growth property, i.e.,
\[\phi_{i,t}(x)\ge\phi_{i,t}(x^{{\rm opt},t,i})+\frac{\gamma_{i,t}}{2}{\rm dist}^2(x,S^{i,t})~~~~{\rm for~all}~x\in[\phi_{i,t}\le\phi_{i,t}^*+\nu_{i,t}],\]
in which $S^{i,t}$ is the set of all optimal points of $\phi_{i,t}$, and $\phi_{i,t}^*=\phi_{i,t}(x^{{\rm opt},t,i})$ 
\end{assume}

 The following regret bound  can be readily deduced from \cite[Corrolary 3.6]{Drusvyatskiy2018}. It is worth noting that the complexity bound aligns with the linear rate of convergence exhibited by the proximal gradient method when employed for strongly convex functions, albeit with a constant factor.
\begin{cor}
The following regret bound holds
\begin{align}
   Re{g_i} = \sum\limits_{t = 1}^T {{\phi _{i,t}}({x^t}) - \sum\limits_{t = 1}^T {{\phi _{i,t}}} ({x^{{\rm{opt}},t,i}})}  \le T\epsilon
\end{align}
after at most 
\begin{align}\label{eq:max t}
    t\le\frac{\beta_T}{2\nu_T}\Gamma_T+12\frac{\beta_T}{\gamma^0}ln\frac{M_T}{\epsilon}~~~{\rm iterations.}
\end{align}
where 
\begin{align}\label{eq:notations}
&M_T=\min_{t\in[T]}\min_{i\in[n]}\phi_{i,t}(x^1)-\phi_{i,t}(x^{{\rm opt},t,i}),~\Gamma_T=\min_{i\in[n]}\min_{t\in[T]}{\rm dist}(x^1,S^{i,t})\\&\nonumber
\beta_T=\min_{i\in[n]}\min_{t\in[T]}L_{f_{i,t}}, ~\nu_{T}=\min_{i\in[n]}\min_{t\in[T]},~\gamma^0=\min_{i\in[n]}\min_{t\in[T]}\gamma_{i,t},~\nu_T=\min_{i\in[n]}\min_{t\in[T]}\nu_{i,t}.
\end{align}
\end{cor}
\begin{proof}
First, we note that by considering Theorem 3.2 and Corollary 3.6  of \cite{Drusvyatskiy2018}, one can see for all $i,t$ that 
\begin{align}\label{eq:quadratic}
    \phi_{i,t}(x)-\phi_{i,t}(x^{{\rm opt},t,i})\le\epsilon~~~\forall x\in[\phi_{i,t}\le\phi^*_{i,t}+\nu_{i,t}]
\end{align}
for \[t\le\frac{L_{f_{i,t}}}{\nu_{i,t}}{\rm dist}(x^1,S^{i,t})+12\frac{L_{f_{i,t}}}{\gamma_{i,t}}Ln(\frac{\phi_{i,t}(x^1)-\phi_{i,t}(x^{{\rm opt},t,i})}{\epsilon}).\]
Now, assume that $x^t$ is generated by Algorithm \ref{alg:PoEG}. Taking into account  \eqref{eq:notations}, and also taking summation from $t=1$ to $t=T$ from \eqref{eq:quadratic}, one observes that $Reg_i\le T\epsilon$ for at most $t$ defined in \eqref{eq:max t}.
\end{proof}

 
 

\section{Conclusions}
We have studied a time-varying multi-objective optimization problem in a setting, which has not been considered previously. We have shown  properties of a natural, online proximal-gradient algorithm when the processing power between two arrivals of new information is bounded.
Going forward, one could clearly consider alternative uses of the same algorithm (e.g., how many operations one requires per update to achieve a certain bound in terms of the dynamic regret), variants of the algorithm, or completely novel settings. 
In parallel with our work, Tarzanagh and Balzano \cite{Tarzanagh2024a} studied online bilevel optimization under assumptions of strong convexity throughout, which could be seen as one such novel setting.

\paragraph*{Acknowledgements}\quad 
This work has received funding from the European Union’s Horizon Europe research and innovation programme under grant agreement No. GA 101070568.

\clearpage
\bibliographystyle{plain} 
\bibliography{references} 


\end{document}